\documentclass{amsart}

\numberwithin{equation}{section}

\usepackage{enumerate}

\usepackage[english]{babel}
\usepackage{amsfonts, amsmath, amsthm, amssymb,amscd,indentfirst,mathrsfs}

\usepackage{bm}
\usepackage{color}
\usepackage{xcolor}
\usepackage{tikz}
\usepackage{esint}

\usepackage[utf8]{inputenc} 
\usepackage[T1]{fontenc}

\usepackage{graphicx}
\usepackage{epstopdf}
\usepackage{latexsym}

\usepackage{palatino}
\usepackage{marginnote}

\theoremstyle{plain}
\newtheorem{theorem}{Theorem}[section]

\newtheorem{proposition}[theorem]{Proposition}

\newtheorem{corollary}[theorem]{Corollary}
\theoremstyle{definition}
\newtheorem{definition}[theorem]{Definition}

\newtheorem{example}[theorem]{Example}

\newtheorem{remark}[theorem]{Remark}

	\def\bp{\begin{proof}}
		\def\ep{\end{proof}}

	\def\dirac{{\partial\!\!\!/}}

	\usepackage{amsaddr}

\begin{document}
		
		\title[Elliptic operators in conical manifolds]{Mapping properties of geometric elliptic operators in conformally conical spaces: an introduction with examples}
		
		\author{Levi Lopes de Lima}
		\address{Universidade Federal do Cear\'a (UFC),
			Departamento de Matem\'{a}tica, Campus do Pici, Av. Humberto Monte, s/n, Bloco 914, 60455-760,
			Fortaleza, CE, Brazil.}

		\email{levi@mat.ufc.br}
		
		\thanks{The author has been supported by CNPq/Brazil grant
			312485/2018-2 and FUNCAP/CNPq/PRONEX grant 00068.01.00/15.
		}
		
		\begin{abstract}
			In this largely expository note, we discuss the mapping properties of the Laplacian (and other geometric elliptic operators) in spaces with an isolated  conical singularity following the approach developed by B.-W. Schulze and collaborators. 
		Our presentation aims at illustrating the 
		versatility of these results by describing how certain representative (and seemingly disparate) applications in Geometric Analysis follow from a common setup. 
		\end{abstract}

		\maketitle


			\section{Introduction}\label{intro}	
		
		The theory of elliptic differential operators acting on sections of a vector bundle over a compact  manifold $X$ is a well established discipline \cite{chazarain2011introduction,hormander2007analysis,wloka1995boundary,grubb2012functional}. 
		If $X$ is boundaryless then we may resort to the fact that any such manifold can  be infinitesimally identified  to euclidean space around each of its points in order to transplant the symbolic calculus of pseudo-differential operators in flat space to this ``curved'' arena.
		As a consequence, 
		the main technical result in the theory is proved, namely, the existence of a (pseudo-differential) param\-etr\-ix for the given elliptic differential operator $D$, from which the  standard mapping properties (regularity of solutions, Fredholmness, etc) in the usual scale of Sobolev spaces may be readily derived. From this perspective, we may assert that the resulting theory is a natural outgrowth of Fourier Analysis as applied to the classical procedure of ``freezing the coefficients''.  
		
		In case the underlying manifold $X$ carries a boundary $\partial X$, a fundamentally distinct approach is needed as the local identification to euclidean space obviously fails to hold in a neighborhood of a point in the boundary (from this standpoint, we are forced to view $\partial X$ as the ``singular locus'' of $X$). We may, however, pass to the double of $X$, say $2X$, and assume that a suitable {\em elliptic} extension of the original operator $D$, say $2D$, is available. Since $\partial(2X)=\emptyset$, we have at our disposal a parametrix for $2D$ which may be employed to construct a pseudo-differential projection $C$ acting on sections restricted to ${\partial X}$ (the Calder\'on-Seeley projector). If the differential operator $B$ defining the given boundary conditions is such that the principal symbol of $A=BC$ is sufficiently non-degenerate (for instance, if the pair $(D,B)$ satisfies the so-called Lopatinsky-Shapiro condition) then a parametrix for $A$ is available (recall that $\partial(\partial X)=\emptyset$) and from this 
		we may deduce the expected mapping properties
		of  the associated boundary value map 
		\begin{equation}\label{bd:map}
			\mathcal Du=(Du,Bu|_{\partial X})
		\end{equation}        
		acting on suitable Sobolev spaces. Thus, the theory of elliptic boundary value problems ultimately hinges  on the fact that the corresponding singular locus $\partial X$ not only is intrinsically smooth but also may be easily ``resolved'' after passage to $2X$\footnote{This trick of passing from the bordered manifold $X$ to the boundaryless manifolds $\partial X$ and $2X$ is a key ingredient in index theory \cite{atiyah1964index}.}.
		
		We may now envisage a situation where the underlying space $X$ displays a singular locus $Y$ which fails to admit such a simple resolution (as a boundary does). For instance, we may agree that {the singular locus $Y\subset X$ has the structure of a smooth closed manifold and that a neighborhood $U\subset X$ of $Y $ is the total space of a fiber bundle} 
		\[
		\begin{array}{ccc}
			\mathcal C^F	& \hookrightarrow&	U \\
			&&
			\,\,\,\,\Big\downarrow \pi\\
			&&	Y
		\end{array}
		\]       
		whose typical fyber is a cone $\mathcal C^F$ over a closed manifold $F$. For simplicity, we assume that this bundle is trivial, so that $U$ carries natural coordinates $(x,y,z)$, where $y\in Y$, $z\in F$ and $x$ {is the radial function obtained after identifying the cone generatrix to the inverval $[0,\delta]$, $\delta>0$,  with $x=0$ along $Y$}. We 
		now infuse a bit of geometry in this discussion by
		requiring that the smooth locus $X'=X\backslash Y$ carries a Riemannian metric $\overline g$ so that 
		\begin{equation}\label{met:edge:ex}
			\overline g|_{U'}=dx^2+x^2g_F(z)+g_Y(y), \quad U'=U\backslash Y,
		\end{equation}
		where $g_F$ and $g_Y$ are fixed Riemannian metrics on $F$ and $Y$, respectively. 
		{By abuse of language, we say that $x$ is a ``defining function'' for $Y$ (with respect to $\overline g$).}
		
		The simplest of such ``edge-type'' manifolds occurs when $Y$ collapses into a point, so we obtain a conical manifold (see Definition \ref{conic:metric} below). In any case, we are led to consider {\em geometric} differential operators (i.e. naturally associated to $\overline g$ such as the Laplacian acting on functions, the Dirac operator acting on spinors, etc.) and pose the general problem of studying their mapping properties in suitable functional spaces. The main purpose of this note is to illustrate through examples how useful this elliptic analysis on singular spaces turns out to be. 
		
		The problem remains of transplanting the highly successful ``smooth'' elliptic theory outlined above to this setting. Clearly, the edge-type structure around $Y$ {poses} an obvious obstruction to a straightforward extension of the pseudo-differential calculus. Indeed, this leads us to suspect that, besides the standard ellipticity assumption on $X'$, a complementary notion of ellipticity around $Y$ is required in order to construct a global parametrix. In this regard, there is  no canonical choice and the final  formulation depends on which technique  one is most familiar with. In the rather informal (and simplified) exposition below, which actually emphasizes the conical case (so that $Y=\{q\}$), we roughly follow the approach developed by B.-W. Schulze and collaborators \cite{schulze1998boundary,egorov2012pseudo}, as we believe it displays an adequate balance between technical subtlety and conceptual transparency. In this setting, a key ingredient is the classical Mellin transform, which allows us to pass from the restriction to $U'$ of the given geometric operator $D$ to its {\em conormal symbol} $\xi_D$. The appropriate complementary notion of ellipticity is then formulated by fixing $\beta\in\mathbb R$ and then  requiring that $\xi_D$, viewed as a polynomial function  {whose coefficients are differential operators} acting on the fiber $F$, is invertible when restricted to the vertical line $\Gamma_\beta=\{z\in\mathbb C;{\rm Re}\,z=\beta\}$\footnote{The moral here is that, when trying to freeze the coefficients of $D$ around the tip of the cone, we are inevitably led to contemplate the Mellin transform as the proper analogue of the Fourier transform which, as already noted, does a perfectly good job in the smooth locus.}. Armed with this notion of ellipticity, an appropriate pseudo-differential calculus may be conceived which leads to 
		the construction of the sought-after parametrix; this has as a formal consequence the Fredholmness of    
		$D=D_\beta$ when acting on the so-called {Sobolev-Mellin scale} $\mathcal H^{\sigma,p}_{\beta}(X)$, {\em independently} of $(\sigma,p)\in\mathbb R\times\mathbb Z_+$. Moreover, the index of $D_\beta$ jumps precisely at those values of $\beta$ for which ellipticity fails by an integer quantity depending on the kernel of $\xi_D|_{\Gamma_\beta}$. For instance, if $D=\Delta$, the Laplacian of the underlying conical metric, which is our main concern here, then $\xi_{\Delta}(z)=z^2+bz+\Delta_{g_F}$ for some $b\in\mathbb R$, so a jump occurs at each $\beta$ satisfying the {\em indicial equation} 
		\begin{equation}\label{int:indicial}
			\beta^2+b\beta-\mu=0,
		\end{equation}
		for some $\mu\in {\rm Spec}(\Delta_{g_F})$, and equals the multiplicity of $\mu$ as an eigenvalue. Thus, if a fairly precise knowledge of ${\rm Spec}(\Delta_{g_F})$ is available, the Fredholm index of $\Delta$ in the whole scale $\mathcal H^{s,p}_{\beta}(X)$ can be determined upon computation at a single value of $\beta$.  This final piece of calculation may be carried out by using the fact {that $\Delta$ gives rise to a densely defined, unbounded operator, say $\Delta_\beta$, acting on the Hilbert sector of the scale, namely, $\mathcal H^{\bullet,2}_{\beta}(X)$}. A separate argument, which boils down to identifying the minimal and maximal domains of this operator, then assures the existence of at least a $\beta_0$  such that $\Delta_{\beta_0}$ {has a {\em unique} closed extension (which is necessarily Fredholm).} Usually, $\beta_0$ lies in the interval determined by the indicial roots (the solutions of (\ref{int:indicial})) corresponding to $\mu=0$, so that in case $b\neq 0$, $\Delta_\beta$ turns out  {to be Fredholm with the {\em same} index as long as $\beta$ varies in the interval with endpoints $0$ and $-b$}; compare with Theorem \ref{self:adj}.

		A detailed presentation of the program outlined above for a general elliptic operator is far beyond the scope of this introductory note. Instead, we merely sketch the argument for the Laplacian in conical manifolds (Sections \ref{conic} and \ref{proof:m:t}) and indicate how the method can be extended to other geometric operators by considering the case of the Dirac operator (Section \ref{map:dirac}). In fact, here we focus instead on illustrating the versatility of this theory by 
		including a few representative applications of these mapping properties in Geometric Analysis (Sections \ref{examp:op} and \ref{nonempty:bd}). We insist, however, that the material discussed here is standard, drawn from a number of sources, so no claim is made regarding originality (except perhaps for the naive computations leading to Theorem \ref{albin:mellin:enh}). Indeed, this note has been written in the perspective that,  
		after reading our somewhat informal account  of a noticeably difficult subject, the diligent reader will be able to fill the formidable gaps upon consultation of the original sources. In  
		this regard, we note that, alternatively to the path just outlined, the mapping results described below may be obtained as a consequence of the powerful ``boundary fibration calculus''   \cite{melrose1993atiyah,melrose1990pseudodifferential,mazzeo1991elliptic,lauter2003pseudodifferential,grieser2001basics,melrose1996differential,gil2007geometry,krainer2018friedrichs} (a comparison of Melrose's $b$-calculus and Schulze's cone algebra appears in \cite{lauter2001pseudodifferential}). Also, direct approaches, which in a sense avoid the consideration of the corresponding pseudo-differential formalism, are available in each specific application we consider here \cite{almaraz2014positive,andersson1993elliptic,bartnik1986mass,chaljub1979problemes,lockhart1985elliptic,lee2006fredholm,almaraz2021spacetime,pacini2010desingularizing}.  We believe, however, that a presentation of their mapping properties  as a repertory of results stemming from a common source  contributes to highlight the unifying features of geometric differential operators in singular spaces.

		\section{Fredholmness of the Laplacian in conformally conical manifolds}\label{conic}

		In this section, we define the class of conformally conical manifolds (this entails a slight modification of (\ref{met:edge:ex}) which incorporates a conformal factor involving a suitable power of the defining function $x$) and discuss a few representative examples in this category. We then introduce the relevant functional spaces (the Sobolev-Mellin scale $\mathcal H^{\sigma,p}_\beta$) and then formulate a result (Theorem \ref{self:adj}) which precisely locates the set of values of $\beta$ for which the corresponding Laplacian is Fredholm with an explictly computable index.    
		
		\subsection{Conformally conical manifolds}\label{con:man} 
		Given a closed Riemannian manifold $(F,g_F)$ of 
		dimension\footnote{{In fact, the general theory also works fine for $n=2$ and the assumption $n\geq 3$ is only needed for Theorem \ref{self:adj} and its consequences.}} $n-1\geq 2$, we consider the {\em infinite cone} $(\mathcal C^{(F,g_F)},g_{\mathcal C,F})$ over $(F,g_F)$: 
		\[
		\mathcal C^{(F,g_F)}={\mathbb R_{>0}\times F}
		\]
		endowed with the cone  metric
		\begin{equation}\label{cone:met}
			g_{\mathcal C,F}=dr^2+r^2g_F, \quad r\in\mathbb R_{>0}. 
		\end{equation}
		We then define the {\em truncated cones} by
		\[
		\mathcal C^{(F,g_F)}_0=\{(r,z)\in \mathcal C^{(F,g_F)};z\in F, 0<r<1\}
		\]
		and 
		\[
		\mathcal C^{(F,g_F)}_\infty=\{(r,z)\in \mathcal C^{(F,g_F)};z\in F, 1<r<+\infty\}, 
		\]
		both endowed with the induced metric.
		We also consider the {\em infinite cylinder} $(\mathsf C^{(F,g_F)},g_{\mathsf C, F})$ over $(F,g_F)$: 
		\[
		\mathsf C^{(F,g_F)}=\mathbb R\times F
		\]
		endowed with the product metric
		\[
		g_{\mathsf C,F}=dr^2+g_F.
		\]
		
		We now consider a compact topological space $X$ which is smooth everywhere except possibly at a point, say $q$.   We endow the smooth locus $X':=X\backslash\{q\}$, $\dim X'=n\geq 3$, with a Riemannian metric $\overline g$ and assume that there  exists a neighborhood $U$  of $q$ (the conical region) such that $U':=U \backslash \{q\}$ is diffeomorphic 
		to $\mathcal C^{(F,g_F)}_0$ and 
		\begin{equation}\label{meg:ov:g}
			\overline g|_{U'}=g_{\mathcal C,F}=dx^2+x^2g_F,
		\end{equation}  
		where for convenience we have set $x=r$ in the description of the cone metric to emphasize that $x$ is viewed as a defining function for $\{q\}$; compare with (\ref{cone:met}).

		\begin{definition}\label{conic:metric}
			A {\em conformally conical manifold} is a pair $(X,g_s)$, where $X$ is as above, and $g_s$ is a Riemannian metric in $X'$ which, restricted to $U'$, satisfies 
			\begin{equation}\label{met:edge}
				g_s:=x^{2s-2}{\left(\overline g+ o(1)\right)},\quad s\in\mathbb R, 
			\end{equation}
			as $x\to 0$.
			We then say that $(F,g_F)$ is the {\em link}
			of  $(X,g_s)$.
		\end{definition}
		
		\begin{remark}\label{flex:conf}
			{Our terminology is justified by the presence of the conformal factor next to $\overline g+o(1)$, which allows us to arrange the examples below in a single geometric structure.}   
		\end{remark}
		
		\begin{remark}\label{decay}
			In applications, it is often needed to append  decay relations to (\ref{met:edge}) for the corresponding derivatives up to second order at least; see Remark \ref{order} below. 
		\end{remark}
		
		\begin{remark}\label{complete}
			$(X',g_s)$ is complete if and only if $s\leq 0$.
		\end{remark}
		
		We will be interested in doing analysis in the open manifold $(X',g_s)$. More precisely, we will study the mapping properties of the Laplacian $\Delta_{g_s}$ in an appropriate scale of Sobolev spaces. Before proceeding, however, we discuss a few examples, which highlight the distinguished roles played by the ``rigid'' spaces $\mathcal C^{(F,g_F)}_0$, $\mathcal C^{(F,g_F)}_\infty$ and $\mathsf C^{(F,g_F)}$  
		as asymptotic models.  
		
		\begin{example}\label{ex:conic} (${\rm AC}_0$ manifolds)
			Let  
			$(V,h)$ be an open manifold for which there exists a compact  $K\subset V$ and a diffeomorphism $\psi: \mathcal C^{(F,g_F)}_0\to V\backslash K$ such that, as $r\to 0$,
			\[
			|\nabla_b^k(\psi^*h-g_{\mathcal C,F})|_b=O(r^{\nu_0-k}), \quad 0\leq k\leq m.
			\]
			Here, $m\geq 0$ is the order and $\nu_0>0$ is the rate of decay. {Also, the subscript $b$ refers to invariants attached to the ``rigid'' conical metric in the model space (the same notation is used in the examples below).} We then say that $(V,h)$ is an {\em asymptotically conical manifold at the origin} (${\rm AC}_0$). Clearly, if we take $x=r$, this corresponds to a {conformally} conical manifold with $s=1$ in (\ref{met:edge}).
		\end{example} 
		
		\begin{example}\label{ex:af}
			(${\rm AC}_\infty$ manifolds)
			Let  
			$(V,h)$ be an open manifold for which there exists a compact  $K\subset V$ and a diffeomorphism $\psi:\mathcal C^{(F,g_F)}_\infty\to V\backslash K$ such that, as $r\to +\infty$,
			\[
			|\nabla_b^k(\psi^*h-g_{\mathcal C,F})|_b=O(r^{-\nu_\infty-k}), \quad 0\leq k\leq m.
			\]
			Here, $m\geq 0$ is the order and $\nu_\infty>0$ is the rate of decay. We then say that $(V,h)$ is an {\em asymptotically conical manifold at infinity} $({\rm AC}_\infty)$. Clearly, if we take $x=r^{-1}$, this corresponds to a {conformally} conical manifold with $s=-1$ in (\ref{met:edge}).
		\end{example}

		\begin{example}\label{asym:cyl}
			(${\rm ACyl}$ manifolds)
			Let  
			$(V,h)$ be an open manifold for which there exists a compact  $K\subset V$ and a diffeomorphism $\psi:\mathsf C^{(F,g_F)}_\infty\to V\backslash K$ such that, as $r\to +\infty$,
			\[
			|\nabla_b^k(\psi^*h-g_{\mathsf C,F})|_b=O(e^{-(\nu_c+k)r}), \quad 0\leq k\leq m.
			\]
			Here, $m\geq 0$ is the order and $\nu_c>0$ is the rate of decay. We then say that $(V,h)$ is an {\em asymptotically cylindrical manifold} $({\rm ACyl})$. Clearly, if we take $x=e^{-r}$, this corresponds to a conformally conical manifold with $s=0$ in (\ref{met:edge}). These manifolds play a central role in the formulation and proof of the Atiyah-Patodi-Singer index theorem \cite{atiyah1975spectral,melrose1993atiyah}.
		\end{example}

		\begin{example}
			\label{a0:ainf} (${\rm AC}_0/{\rm AC}_\infty$ manifolds) Assume more generally that $V\backslash K$ decomposes as a {\em finite} union of ends which are either ${\rm AC}_0$ or ${\rm AC}_\infty$. These manifolds, which are called {\em conifolds} in \cite{pacini2013special}, appear prominently in the study of moduli spaces of special Lagrangian submanifolds; see also \cite{joyce2003special}. 
		\end{example}

		\begin{remark}\label{order} 
			In all examples above, we take $m\geq 2$.
		\end{remark}
		
		\subsection{Sobolev-Mellin spaces and Fredholmness}\label{mel:sob:fred}
		Given  $\beta\in \mathbb R$ and integers $k\geq 0$ and $1< p<+\infty$, we define $\mathcal H_\beta^{k,p}(X)$ to be the space of all distributions $u\in L^p_{\rm loc}(X',d{\rm vol}_{\overline g})$   such that:
		\begin{itemize}
			\item for any cutoff function $\varphi$ with $\varphi\equiv 1$ near $q$ and $\varphi\equiv 0$ outside $U$, we have that $(1-\varphi)u$ lies in the standard Sobolev space $H^{k,p}(X',d{\rm vol}_{\overline g})$;
			\item there holds
			\begin{equation}\label{sob:def:con}
				x^{\beta}\mathsf D^j\partial_z^\alpha(\varphi u)(x,z)\in L^p(X',d_+xd{\rm vol}_{g_F}), \quad j+|\alpha|\leq k.
			\end{equation}
			Here, $\mathsf D=x\partial_x$ is the Fuchs operator and $d_+x=x^{-1}dx$. 
		\end{itemize}
		Using duality and interpolation, we may define $\mathcal H_\beta^{\sigma,p}(X)$ for any $\sigma\in\mathbb R$. As usual, $\mathcal H_\beta^{\sigma,p}(X)$ is naturally a Banach space which is Hilbert for $p=2$. 
		For instance, when $k=0$ the corresponding norm {to the $p^{\rm th}$ power} reduces to the integral
		\begin{equation}\label{norm:near}
			\int|x^\beta u(x,z)|^pd_+xd{\rm vol}_{g_F}(z)
		\end{equation}
		near $q$. 
		These are the weighted Sobolev-Mellin spaces considered in \cite{schrohe1999ellipticity}, except that there they are labeled by 
		\begin{equation}\label{beta:gamma}
			\gamma=\frac{n}{2}-\beta.
		\end{equation} 
		In order to confirm the scale character of these spaces, we recall the relevant embedding theorem; see \cite[Remark 2.2]{coriasco2007realizations} and 	\cite[Corollary 2.5]{roidos2013cahn}.
		
		\begin{proposition}\label{sob:emb}
			One has a continuous embedding $\mathcal H_{\beta'}^{\sigma',p}(X)\hookrightarrow \mathcal H_\beta^{\sigma,p}(X)$ if $\beta'{\leq}\beta$ and $\sigma'{\geq}\sigma$, which is compact if the strict inequalities hold. {Also,
				if ${\sigma}>n/p$ then any $u\in \mathcal H_\beta^{\sigma,p}(X,g)$ is continuous in $X'$ and satisfies $u(x)=O(x^{-\beta})$ as $x\to 0$.}
		\end{proposition}

		It is clear that the Laplacian $\Delta_{g_s}$ defines a bounded map
		\begin{equation}\label{lap:cont:sob}
			\Delta_{g_s,\beta}:\mathcal H_\beta^{\sigma,p}(X)\to \mathcal H_{\beta+2s}^{\sigma-2,p}(X), 
		\end{equation} 
		and our primary concern here is to study its mapping properties.
		As already discussed in the Introduction, we should be aware that a key point in the analysis of an elliptic operator in a conformally conical manifold is that, differently from what happens in the smooth case, invertibility of its principal symbol does not suffice to make sure that a parametrix exists. In particular, it is not clear whether (\ref{lap:cont:sob}) is Fredholm for some value of the weight $\beta$. It turns out that this Fredholmness property  is  insensitive to the pair $(\sigma,p)$ but depends crucially on $\beta$ \cite{schrohe1999ellipticity}. Indeed, it turns out that this map is Fredholm for all but a discrete set of values of $\beta$, with the index possibly jumping only when $\beta$ reaches these exceptional values.  
		We now state a useful result that confirms this expectation for the map (\ref{lap:cont:sob}).
		For this, we introduce the quantity 
		\begin{equation}\label{exp:a}
			a=(n-2)s.
		\end{equation}
		If $s\neq 0$ then $a\neq 0$ as well if we further assume that $n\geq 3$. We then  denote by $I_a$ the open {interval with endpoints $a$ and $0$}.

		\begin{theorem}\label{self:adj}
			If $n\geq 4$ and $a\neq 0$ then the Laplacian map $\Delta_{g_s,\beta}$ in (\ref{lap:cont:sob}) is Fredholm of index $0$ whenever $\beta\in I_a$.  
		\end{theorem}
		
		As already remarked, from this we can read off the Fredholm index of $\Delta_{g_s,\beta}$ as $\beta$ varies if a complete knowledge of the spectrum of $\Delta_{g_F}$ is available. As another useful application of Theorem \ref{self:adj}, we mention the following existence result, which is just a restatement of Fredholm alternative. 
		
		\begin{corollary}
			If $\beta\in I_{a}$, $a\neq 0$, then the map (\ref{lap:cont:sob}) is surjective if and only if it is injective. 
		\end{corollary}

		\begin{remark}\label{a:0}
			The case $a=0$ may also be treated by the method leading to Theorem \ref{self:adj}. It turns out that $\Delta_{g_0,\beta}$ is Fredholm for any $\beta$ such that $\beta^2\notin {\rm Spec}(\Delta_{g_F})$; see Remark \ref{a:0:2}. 
		\end{remark}

		\section{The proof of Theorem \ref{self:adj} (a sketch)}\label{proof:m:t}
		
		Our aim here is to sketch  the proof of Theorem \ref{self:adj}. This  may be confirmed in a variety of ways on inspection of standard sources; see for instance \cite{melrose1993atiyah,mazzeo1991elliptic,schulze1998boundary,lesch1997differential,egorov2012pseudo,melrose1996differential}, among others. However, since in these references the arguments leading to Theorem \ref{self:adj} appear embedded in rather elaborate theories, we include a sketch of the proof here in the setting of the Sobolev-Mellin spaces introduced above. 
		In fact, this section may be regarded as an essay on these fundamental contributions as applied to a rather simple situation.

		Since $\Delta_{g_s}$ is elliptic on $X'$, a local parametrix may be found in this region by standard methods. Thus, analyzing the mapping properties of $\Delta_{g_s}$ involves the consideration of a suitable notion of ellipticity in the conical region $U'$. 	Starting from (\ref{meg:ov:g}) and (\ref{met:edge}), we easily compute that 
		the Laplacian $\Delta_{g_s}$ satisfies 
		\begin{equation}\label{conorm}
			P:=x^{2s}\Delta_{g_s}|_{U'}=\mathsf D^2+a\mathsf D+\Delta_{g_F}+o(1), 
		\end{equation}
		where $\mathsf D=x\partial_x$. 
		As already noted, the needed ingredients to establish the mapping properties for $\Delta_{g_s}$ include not only its ellipticity when restricted to the smooth locus, but also the invertibility of the so-called {\em conormal symbol}, which is obtained by freezing the coefficients of $P$ at $x=0$, that is, passing to 
		\begin{equation}\label{con:symb:free}
			P_0=\mathsf D^2+a\mathsf D+\Delta_{g_F}, 
		\end{equation}
		and then applying the Mellin transform $\mathsf M$; see \cite{schrohe1999ellipticity,schulze1998boundary,egorov2012pseudo} and also (\ref{conor:symb}) below,  where this construction is actually applied to an appropriate conjugation of $P_0$.
		Recall that  $\mathsf M$
		is the linear map that to each well-behaved function $f:\mathbb R_+\to\mathbb C$ associates another function $\mathsf M(f):U_f\subset \mathbb C\to \mathbb C$ by means of 
		\[
		\mathsf M(f)(\zeta)=\int_0^{+\infty}f(x)x^{\zeta}d_+x, \quad d_+x=x^{-1}dx.
		\]
		For our purposes, it suffices to know that this transform  
		meets the following properties: 
		\begin{itemize} 
			\item For each $\theta\in\mathbb R$, the map 
			\[
			x^\theta L^2(\mathbb R_+,d_+x)\stackrel{\mathsf M}{\longrightarrow} 
			L^2(\Gamma_{{-\theta}}),
			\]
			is an isometry.
			{Here, $\Gamma_{{\alpha}}=\{\zeta\in\mathbb C;{\rm Re}\,\zeta={\alpha}\}$, $\alpha\in\mathbb R$, and $x^\theta L^2(\mathbb R_+,d_+x)$ is endowed with the inner product
				\begin{equation}\label{inner:prod}
					\langle u,v\rangle_{x^\theta L^2(\mathbb R_+,d_+x)}=\langle x^{-\theta}u,x^{-\theta}v\rangle_{L^2(\mathbb R_+,d_+x)}.
			\end{equation}}
			Moreover, each element $u$ in the image extends holomorphically to the half-space {$\{\zeta\in\mathbb C;{\rm Re}\,\zeta>{-\theta}\}$} (Notation: $u\in\mathscr H(\{{\rm Re}\,\zeta>{-\theta}\})$).
			\item $\mathsf M(\mathsf Df)(\zeta)=-\zeta\mathsf M(f)(\zeta)$.  
		\end{itemize}
		In particular, the conormal symbol
		\begin{equation}\label{con:symb:lap}
			\xi_{\Delta_{g_s}}(\zeta)=\zeta^2-a\zeta+\Delta_{g_F}
		\end{equation}
		is obtained by Mellin tranforming (\ref{con:symb:free}). Note that this is a polynomial function with coefficients in the space of differential operators on the link $(F,g_F)$.  
		
		\begin{definition}\label{def:ellip:con}
			The Laplacian $\Delta_{g_s}$ is {\em elliptic} (with respect to some $\beta\in\mathbb R$) if 
			\[
			\xi_{\Delta_{g_s}}(\zeta):H^{\sigma,p}(F,d{\rm vol}_{g_F})\to H^{\sigma-2,p}(F,d{\rm vol}_{g_F})
			\]
			is invertible for any $\zeta\in\Gamma_\beta$. Here, $H^{\sigma,p}$ denotes the standard Sobolev scale.
		\end{definition}

		\begin{remark}\label{rem:gen:geo}
			Inherent in the discussion above is the fact that the Laplacian can be written as a polynomial in $\mathsf D$ in the conical region. More generally, we may consider any elliptic operator $D$ satisfying, as $x\to 0$,  
			\[
			x^\nu D|_{U'}=\sum_{i=0}^{m}A_i(x)\mathsf D^i+o(1),\quad \nu>0,
			\] 
			where each $A_i(x)$ is a differential operator of order at most $m-i$ acting on (sections of a vector bundle over) $F$ \cite{schulze1998boundary,lesch1997differential}. 
			Definition \ref{def:ellip:con} then applies to the corresponding 
			conormal symbol, which  is 		\[
			\xi_D(\zeta)=\sum_{i=0}^m(-1)^iA_i(0)\zeta^i.
			\]
			Besides the Laplacian, in next section we consider another most honorable example, namely, the Dirac operator acting on spinors. 
		\end{remark}

		Armed with this notion of ellipticity, we may setup an appropriate pseudo-differential calculus that enables the construction of a parametrix for $\Delta_{g_s}$ in the Sobolev-Mellin scale $\mathcal H^{\sigma,p}_\beta(X)$; the quite delicate argument can be found in \cite{schulze1998boundary,egorov2012pseudo}. As in the smooth case, this turns out to be formally equivalent to the assertion that the map (\ref{lap:cont:sob}) is Fredholm. 
		
		\begin{remark}\label{rem:conv}
			The converses in the chain of implications above also hold true, so that (\ref{lap:cont:sob}) fails to be Fredholm precisely at those $\beta$ for which the invertibility condition fails. More precisely, if we set 
			\begin{equation}\label{ind:roots:up}
				\Xi_{\beta}:=\left\{\zeta\in\mathbb C;\zeta^2-a\zeta-\mu=0, \mu\in{\rm Spec}(\Delta_{g_F})\right\}\cap\Gamma_{\beta}.
			\end{equation}
			then the Laplacian map in (\ref{lap:cont:sob}) fails to be  Fredholm if and only if $\Xi_{\beta}\neq\emptyset$. This takes place along the discrete set formed by those $\beta=\beta_\mu$ satisfying the {\em indicial equation}
			\[
			\beta_\mu^2-a\beta_\mu-\mu=0, \quad \mu\in {\rm Spec}(\Delta_{g_F}),
			\] 
			and a further argument shows that the corresponding jump in the Fredholm index equals 
			\begin{equation}\label{jump:fac}
				\pm\dim\ker \xi_{\Delta_{g_s}}(\beta_\mu)=\pm\dim\ker (\Delta_{g_F}+\mu).
			\end{equation} 
		\end{remark}	
		
		From the  previous remark, a first step toward computing the Fredholm index of $\Delta_{g_s,\beta}$ as $\beta$  varies involves first determining it at a single value of $\beta$. 
		A possible approach to this goal is to 
		consider the {\em core} Laplacian
		\begin{equation}\label{core:lap}
			(\Delta_{g_s},C^\infty_c(X')):C^\infty_c(X')\subset \mathcal H_\beta^{0,2}(X)\to \mathcal H_\beta^{0,2}(X),
		\end{equation}
		a  densely defined operator  whose closure is the operator $(\Delta_{g_s},D_{\rm min}(\Delta_{g_s}))$, with domain $D_{\rm min}(\Delta_{g_s})$  formed by those
		$u\in \mathcal H_\beta^{0,2}(X)$ such that there exists $\{u_n\}\subset C^\infty_c(X')$ with  $u_n{\to} u$  
		{and} $\{\Delta_{g_s}u_n\}$ is Cauchy in $\mathcal H_\beta^{0,2}(X)$.
		Also, we may consider $(\Delta_{g_s},D_{\rm max}(\Delta_{g_s}))$, where 
		\[
		D_{\rm max}(\Delta_{g_s})=\left\{u\in \mathcal H_\beta^{0,2}(X);\Delta_{g_s}u\in \mathcal H_\beta^{0,2}(X)\right\}.
		\] 
		Regarding these notions, the following facts are well-known.
		
		\begin{itemize}
			\item $D_{\rm min}(\Delta_{g_s})\subset D_{\rm max}(\Delta_{g_s})$;
			\item If $(\hat\Delta_{g_s},{\rm Dom}(\hat\Delta_{g_s}))$ is a closed extension of $(\Delta_{g_s},C^\infty_c(X'))$ then
			\[
			D_{\rm min}(\Delta_{g_s})\subset {\rm Dom}(\hat\Delta_{g_s})\subset  D_{\rm max}(\Delta_{g_s}).
			\]  
		\end{itemize}
		Hence, in order to understand the set of closed extensions, we need to look at the subspaces of  	
		the {\em asymptotics space} 
		\begin{equation}\label{as:space}
			\mathcal Q(\Delta_{g_s}):=\frac{D_{\rm max}(\Delta_{g_s})}{D_{\rm min}(\Delta_{g_s})}.
		\end{equation}
		Thus, $\mathcal Q(\Delta_{g_s})=\{0\}$ implies that the Laplacian {has a unique closed extension and hence the associated map (\ref{lap:cont:sob}) is Fredholm. In particular, it is essentially self-adjoint (hence with a vanishing index) whenever it is symmetric}. From this, the remaining values of the index as $\beta$ varies  may be determined by means of the jump factors in (\ref{jump:fac}).

		The properties of the Mellin transform mentioned above suggest to work with the ``Mellin'' volume element 
		\[
		d{\rm vol}_{\mathsf M}=x^{-1}dxd{\rm vol}_{g_F}
		\]
		instead of the  volume element $x^{n-1}dxd{\rm vol}_{g_F}$ associated to $\overline g$. 
		This is implemented by working ``downstairs'' in the  diagram below, where $\tau=x^{\frac{n}{2}}$ is unitary and $\Delta^\tau_{g_s}=\tau\Delta_{g_s}\tau^{-1}$:
		\begin{equation}\label{diag}
			\begin{array}{ccc}
				D_{\rm max}(\Delta_{g_s})\subset \mathcal H_\beta^{0,2}(X) & \xrightarrow{\,\,\,\,\Delta_{g_s}\,\,\,\,} & 	\mathcal H_\beta^{0,2}(X) \\
				\tau	\Big\downarrow &  & \Big\downarrow \tau \\
				D_{\rm max}(\Delta^\tau_{g_s})\subset	{x^{\frac{n}{2}-\beta}}	L^2(X',d{\rm vol}_{\mathsf M}) & \xrightarrow{\,\,\,\,\Delta^\tau_{g_s}\,\,\,\,} & 	{x^{\frac{n}{2}-\beta}}	L^2(X',d{\rm vol}_{\mathsf M}) 
			\end{array}
		\end{equation}
		
		{
			\begin{remark}\label{self-ad-delta}
				It is immediate to check that, near the singularity,
				\[
				\langle \Delta_{g_s}u,v\rangle_{\mathcal H^{0,2}_\beta(X)}=\int x^{2\beta-ns}v\Delta_{g_s}u\, d{\rm vol}_{g_s}, 
				\]   
				so that the horizontal maps in (\ref{diag}) define symmetric operators if and only if $\beta=ns/2$. Notice that the same conclusion holds true for any operator which is formally self-adjoint with respect to $d{\rm vol}_{g_s}$.
			\end{remark}
		}
		
		Let $u\in D_{\rm max}(\Delta_{g_s})$. Thus, $v:=\tau u\in D_{\rm max}(\Delta^\tau_{g_s})$ satisfies $x^{\beta-n/2}v\in L^2(X',d{\rm vol}_{\mathsf M})$, so that 
		$\mathsf M(v)\in{\mathscr H}(\{{\rm Re}\,\zeta>\beta-n/2\})$. On the other hand, if 
		\[
		P_0^\tau:=\tau P_0\tau^{-1}=
		\mathsf D^2+(a-n)\mathsf D+\frac{n(n-2a)}{4}+\Delta_{g_F}, 
		\]
		then $w:=P_0^\tau v$ satisfies $x^{-2s+\beta-n/2}w=x^{\beta-n/2}\tau\Delta_{g_s}u\in L^2(X',d{\rm vol}_{\mathsf M})$, so that $\mathsf M(w)\in{\mathscr H}(\{{\rm Re}\,\zeta> -2s+\beta-n/2\})$.
		By taking Mellin transform,
		\[
		\mathsf M(w)(\zeta,z,y)=\xi_{\Delta^\tau_{g_s}}(\zeta)\mathsf M(v)(\zeta,z,y), 
		\]
		where 
		\begin{equation}\label{conor:symb}
			\xi_{\Delta^\tau_{g_s}}(\zeta)=	\zeta^2+(n-a)\zeta+\frac{n(n-2a)}{4}+\Delta_{g_F}
		\end{equation}
		is the conormal symbol of $\Delta^\tau_{g_s}$.
		The conclusion is that, at least formally, 
		\begin{equation}\label{formally}
			\mathsf M(v)(\zeta,z,y)=\xi_{\Delta^\tau_{g_s}}^{-1}(\zeta)\mathsf M(w)(\zeta,z,y), 
		\end{equation}
		but we should properly handle the zeros of $\xi_{\Delta^\tau_{g_s}}$
		located within the critical strip $\Gamma_{-2s+\beta-n/2,\beta-n/2}$, which we may gather together in the {\em asymptotics set}\footnote{Note that $\Lambda_\beta^\tau\subset\mathbb R$ by (\ref{roots}) and the fact that ${\rm Spec}(\Delta_{g_F})\subset[0,+\infty)$.}
		\[
		\Lambda^\tau_{\beta}:=\left\{\zeta\in\mathbb C;Q_\mu(\zeta)=0,\mu\in{\rm Spec}(\Delta_{g_F})\right\}\cap \Gamma_{-2s+\beta-n/2,\beta-n/2}.
		\]
		Here,  $\Gamma_{c,c'}=\{\zeta\in \mathbb C; c<{\rm Re}\,\zeta<c'\}$ for $c<c'$ and 
		\[
		Q_\mu(\zeta)=\zeta^2+(n-a)\zeta+\frac{n(n-2a)}{4}-\mu.
		\] 
		
		Since the roots of $Q_\mu$ are explicitly given by 
		\begin{equation}\label{roots}
			\frac{a-n}{2}\pm \delta^\pm_{\mu},\quad 	\delta^\pm_{\mu}=\pm\frac{1}{2}\sqrt{a^2+4\mu},
		\end{equation}
		we may alternatively consider  
		\[
		\tilde\Lambda^{\tau,\pm}_{\beta}=\left\{\mu\in{\rm Spec}(\Delta_{g_F});\delta_{\mu}^{\pm}\in\Gamma_{-2s+\beta-a/2,\beta-a/2}\right\}.
		\]	
		After applying Mellin inversion to (\ref{formally}) and using the appropriate pseudo-differential calculus \cite{lesch1997differential,schrohe1999ellipticity,schulze1998boundary}, we obtain
		\begin{equation}\label{expansion}
			v-w=\sum_{\mu\in\tilde\Lambda^{\tau,\pm}_{\beta}}A_\mu(x,z,y),
		\end{equation}	
		where the right-hand side represents a generic element in the asymptotics space $\mathcal Q(\Delta_{g_s})$. Thus, the elements in $\tilde\Lambda_\beta^{\tau,\pm}$ constitute the obstruction to having $v=w$ (and hence, $\mathcal Q(\Delta_{g_s})=\{0\}$).
		From this we easily derive the 
		next results.
		
		\begin{theorem}\label{self:crit}
			The core Laplacian {has a unique closed extension} whenever  $\tilde\Lambda^{\tau,\pm}_{\beta}=\emptyset$. 
		\end{theorem} 
		
		\begin{corollary}\label{map:p:lap}
			Assume that $n\geq 4$. Then the core Laplacian has a unique closed, Fredholm extension if either i) $s\leq 0$  or ii) $s>0$ and  $\beta=ns/2$. {In both cases, it is essentially self-adjoint for $\beta=ns/2$.}
		\end{corollary}
		
		\begin{proof}
			The case $s\leq 0$ follows from the fact that $\Gamma_{-2s+\beta-n/2,\beta-n/2}=\emptyset$, which clearly implies that $\tilde\Lambda^{\tau,\pm}_{\beta}=\emptyset$ as well. 
			If $s>0$ then  
			\[
			|\delta^\pm_\mu|\geq\frac{a}{2}=\frac{(n-2)s}{2}\geq s,
			\]
			so that 
			\[
			\tilde\Lambda^{\tau,\pm}_{ns/2}=\left\{\mu\in{\rm Spec}(\Delta_{g_F});\delta_{\mu}^{\pm}\in\Gamma_{-s,s}\right\}=\emptyset
			\] 
			indeed. {The last assertion follows from Remark \ref{self-ad-delta}}.
		\end{proof}

		In each case of Corollary \ref{map:p:lap}, the corresponding map (\ref{lap:cont:sob}) is Fredholm  and  this turns out to be a crucial step in the proof of Theorem \ref{self:adj}. Indeed, we already know that  Fredholmness and the associated index do not depend on the pair $(\sigma,p)$ but only on $\beta$. The key point now is that, as already explained  in the slightly different (but equivalent) setting of the discussion surrounding Remark \ref{rem:conv},  the strategy to preserve Fredholmness as $\beta$ varies involves precluding the crossing of  zeros of $\xi_{\Delta_{g_s}^\tau}$ through the critical line $\Gamma_{\beta-n/2}$ (this is what ellipticity is all about). Precisely, we consider
		\[
		\Xi^\tau_{\beta}:=\left\{\zeta\in\mathbb C;Q_\mu(\zeta)=0, \mu\in{\rm Spec}(\Delta_{g_F})\right\}\cap\Gamma_{\beta-n/2},
		\] 	
		and the relevant result is that $\Delta_{g_s}$ remains Fredholm with the {\em same} index as long as $\Xi^\tau_{\beta}=\emptyset$; see  \cite[Section 3]{schrohe1999ellipticity} or \cite[Subsection 2.4.3]{schulze1998boundary}. Certainly, this is the case for all $\beta\in I_a$, $a\neq 0$. Since (the closure of) this interval always contains $ns/2$, the proof of Theorem \ref{self:adj} follows from Corollary \ref{map:p:lap} and the remarks above.

		\begin{remark}\label{a:0:2}
			The case $a=0$ follows by a similar argument observing that the roots of $Q_\mu(\zeta)=0$ are 
			$
			-{n}/{2}\pm\sqrt{\mu},
			$
			so that Fredholmness fails whenever $\beta=\pm\sqrt{\mu}$; compare with Remark \ref{a:0} and  \cite[Theorem 6.2]{lockhart1985elliptic}.
		\end{remark}

		\begin{remark}\label{upstairs}
			We emphasize that the authors in \cite{schrohe1999ellipticity} and \cite{schulze1998boundary} work ``upstairs'' in respect to the diagram (\ref{diag}), that is, before applying the conjugation $\tau=x^{n/2}$, so instead of $\Xi^\tau_{\beta}$ they consider $\Xi_{\beta}$ as in (\ref{ind:roots:up}). 
			Notice that the polynomial equation here is the Mellin transform of $P_0$ whereas the critical line is shifted to the right by $n/2$. It is immediate to check that both approaches produce the same numerical results for the Fredholmness of $\Delta_{g_s,\beta}$.
		\end{remark}

		\section{The Dirac operator}\label{map:dirac}
		We now illustrate how flexible the theory described in the previous section is by explaining how it may be adapted to establish the mapping properties of the Dirac operator
		\begin{equation}\label{dirac:map}
			{\dirac}_{g_s}:\mathcal H^{s,p}_\beta(S_X)\to \mathcal H^{s-1,p}_{\beta+1}(S_X)
		\end{equation}
		in the appropriate scale of Sobolev-Mellin spaces. Here, $X$ is assumed to be spin and $S_X$ is the corresponding spinor bundle (associated to $g_s$). As usual, we first consider the core Dirac operator
		\begin{equation}\label{dirac:map:c}
			({\dirac}_{g_s},C^\infty_0(S_X)):C^\infty_0(S_X)\subset \mathcal H^{0,2}_\beta(S_X)\to H^{0,2}_\beta(S_X),
		\end{equation}
		and our aim is to give conditions on $\beta$ to make sure that the associated asymptotics space is trivial. 
		
		It follows from \cite[Lemma 2.2]{albin2016index} that, in the conical region,
		\[
		{\dirac}_{\overline g}={\mathfrak c}(\partial_x)\left(\partial_x+\frac{n-1}{2x}+\frac{1}{x}{\dirac}_F\right)+O(1),
		\]
		where $\mathfrak c$ is Clifford product and ${\dirac}_F$ is the Dirac operator of the spin manifold $(F,g_F)$. From \cite[Proposition 2.31]{bourguignon2015spinorial}, we thus obtain 
		\[
		\dirac_{g_s}=x^{1-s}{\mathfrak c}(\partial_x)\left(\partial_x+\frac{\hat a}{x}+\frac{1}{x}{\dirac}_F\right)+O(1), \quad \hat a=\frac{(n-1)s}{2},
		\]
		so that 
		\[
		\mathscr P:=x^s{\dirac}_{g_s}={\mathfrak c}(\partial_x)\mathscr P_0+O(x),
		\]
		where 
		\[ 
		\mathscr P_0=\mathsf D+\hat a+{\dirac}_F
		\]
		is the conormal symbol.
		By working ``downstairs'', we get
		\[
		\mathscr P_0^\tau:=\tau\mathscr P_0\tau^{-1}=\mathsf D+\hat a-\frac{n}{2}+ {\dirac}_F,
		\] 
		and after Mellin transforming this we see that the corresponding asymptotics  set is 
		\[
		\Theta^\tau_\beta:=\left\{\zeta\in \mathbb C; \zeta+\frac{n}{2}-\hat a-\vartheta=0,\vartheta\in{\rm Spec}({\dirac}_F)\right\}\cap\Gamma_{-s+\beta-n/2,\beta-n/2}.  
		\]
		By arguing exactly as above, we easily obtain the following result.
		
		\begin{theorem}\label{albin:mellin}
			The core Dirac (\ref{dirac:map:c}) { has a unique closed extension} whenever $\Theta^\tau_\beta=\emptyset$. In particular, this happens if either i) $s\leq 0$ or ii) $s>0$, $\beta=n/2$ 
			and the ``geometric Witt assumption''
			\begin{equation}\label{geo:witt}
				{\rm Spec}({\dirac}_F)\cap\left(\frac{n}{2}-\hat a-s,\frac{n}{2}-\hat a\right)=\emptyset
			\end{equation}
			is satisfied. In this latter case, the Dirac map (\ref{dirac:map})
			is Fredholm of index $0$ if $n/2-s< \beta<n/2$, {with the core Dirac being essentially self-adjoint for $s=1$}. 
		\end{theorem}
		
		\begin{remark}\label{albin:cp}
			This should be compared with \cite[Theorem 1.1]{albin2016index}, which proves essential self-adjointness for $s=1$ in the general edge setting. An alternate approach to this latter result, which works more generally for  stratified spaces, has been recently put forward in \cite{hartmann2018domain}.  
		\end{remark}
		
		\begin{remark}\label{cheeger}
			If $D=d+d^*$, the Hodge-de Rham operator acting on differential forms, then the analogue of the Witt condition above translates into a purely topological obstruction. Precisely, if the cohomology group $H^{\frac{n-1}{2}}(F,\mathbb R)$ is trivial (in particular, if $n$ is even) then, after possibly rescaling the link metric $g_F$, $D_{g_1}$ is essentially self-adjoint \cite{cheeger1979spectral}. Extensions of this foundational result to general stratified spaces appear in \cite{albin2012signature}.
		\end{remark}
		
		This Fredholmness property of $\dirac_{g_s}$ may be substantially improved if we assume that $\kappa_{\overline g}$, the scalar curvature of $\overline g$, is non-negative when restricted to the conical region $U'$. Since
		\[
		\kappa_{\overline g}|_{U'}=\left(\kappa_{g_F}-(n-1)(n-2)\right)x^{-2}+O(x^{-1}), \quad x\to 0, 
		\]
		we infer that $\kappa_{g_F}\geq (n-1)(n-2)>0$ and a well-known estimate \cite[Section 5.1]{friedrich2000dirac} gives 
		\[
		\vartheta\in{\rm Spec}(\dirac_{F})\Longrightarrow |\vartheta|\geq \frac{n-1}{2}, 
		\]
		which allows us to replace  (\ref{geo:witt}) by 
		\begin{equation}\label{geo:witt:imp}
			{\rm Spec}({\dirac}_F)\cap\left(\frac{1-n}{2},\frac{n-1}{2}\right)=\emptyset,
		\end{equation}
		{a gap estimate that, remarkably, does not involve the parameter $s$.} 
		In this way we obtain the following specialization  of Theorem \ref{albin:mellin}.
		\begin{theorem}\label{albin:mellin:enh}
			If $|s|\leq 1$ and  $\kappa_{g_s}|_{U'}\geq 0$
			then the Dirac map (\ref{dirac:map})
			is Fredholm of index $0$ whenever 
			\begin{equation}\label{allow:int}
				\frac{1}{2}(n-1)(s-1)< \beta<\frac{1}{2}(n-1)(s+1).
			\end{equation} 
		\end{theorem}
		
		\begin{proof}
			If $\alpha={(s-1)(n-2)}/{2}$, a computation shows that 
			\[
			\kappa_{g_s}|_{U'}=x^{-\frac{\alpha(n+2)}{n-2}}\left((n-1)(n-2)(1-s^2)x^{\alpha-2}+\kappa_{\overline g}|_{U'}x^\alpha\right),
			\]
			so that $\kappa_{g_s}|_{U'}\geq 0$ implies $\kappa_{\overline g}|_{U'}\geq 0$ {and we may appeal to (\ref{geo:witt:imp}) to obtain (\ref{allow:int}) as the interval where the index remains constant. Since $\beta=ns/2$ lies in this interval if and only if $1-n<s<1+n$, the result follows by Remark \ref{self-ad-delta}.} 
		\end{proof}	
		\section{Applications}\label{examp:op}
		
		We now discuss a few (selected) applications of Theorems \ref{self:adj}, \ref{albin:mellin}
		and \ref{albin:mellin:enh} and Remark \ref{a:0:2} in Geometric Analysis.
		
		\subsection{The Laplacian in ${\rm AC}_0$ manifolds} This class of manifolds appears in Example \ref{ex:conic} above, so that $s=1$ in (\ref{met:edge}). Thus, Theorem \ref{self:adj} applies with $h=g_1$ and $a=n-2$.  It is convenient here to pass from $\beta$ to $\gamma$ as in (\ref{beta:gamma}), so  the Sobolev-Mellin norm in (\ref{norm:near}) becomes
		\begin{equation}\label{norm:0}
			\int|x^{\frac{n}{2}-\gamma} u(x,z)|^px^{-1}dxd{\rm vol}_{g_F}(z),
		\end{equation}
		which gives rise to the Sobolev-Mellin spaces $\mathcal H^{\sigma,\gamma}_{p}(V)$ considered in \cite{schrohe1999ellipticity}. The following result is an immediate consequence of Theorem \ref{self:adj}.
		
		\begin{theorem}\label{self:ac0}
			If $n\ge 4$ then the Laplacian map
			\[
			\Delta_{h,\gamma}:\mathcal H^{\sigma,\gamma}_{p}(V)\to \mathcal H^{\sigma-2,\gamma-2}_{p}(V)
			\]
			is Fredholm of index $0$ if $(4-n)/2<\gamma<n/2$.
		\end{theorem}
		
		This result is used in \cite[Section 2]{de2022scalar} as a key step in the argument toward proving that a function which is negative somewhere is the scalar curvature of some conical metric in $V$.   
		
		\subsection{The Laplacian in ${\rm AC}_\infty$ manifolds} This class of manifolds appears in Example \ref{ex:af} above, so that $s=-1$ in (\ref{met:edge}). Thus, Theorem \ref{self:adj} applies with $h=g_{-1}$ and $a=2-n$.  If $r=x^{-1}$  then the Sobolev-Mellin norm in (\ref{norm:near}) becomes
		\begin{equation}\label{norm:1}
			\int|r^{-\beta} u(r,z)|^pr^{-n}d{\rm vol}_{h}(r,z),
		\end{equation}
		which gives rise to the weighted Sobolev spaces $L^p_{\sigma,\beta}(V)$ considered in \cite[Section 9]{lee1987yamabe}. The following result is an immediate consequence of Theorem \ref{self:adj}; compare with \cite[Theorem 9.2 (b)]{lee1987yamabe}.
		
		\begin{theorem}\label{self:ac0the}
			If $n\ge 4$ then the Laplacian map
			\[
			\Delta_{h,\beta}:L^p_{\sigma,\beta}(V)\to L^p_{\sigma-2,\beta-2}(V)
			\]
			is Fredholm of index $0$ if $2-n<\beta<0$.
		\end{theorem}
		
		\begin{remark}\label{as:anal}
			Consider the case in which the link is the round sphere $(\mathbb S^{n-1},\delta)$ and $\nu_\infty>(n-2)/2$. Thus, we are in the {\em asymptotically flat} case so dear to practitioners of Mathematical Relativity \cite{lee1987yamabe,bartnik1986mass}. Here, an asymptotic invariant for $(V,h)$, the ADM mass $\mathfrak m_{(V,h)}$, is defined by
			\[
			\mathfrak m_{(V,h)}=\lim_{r\to +\infty}\int_{S^{n-1}_r}\left(h_{ij,j}-h_{jj,i}\right){\eta}^idS^{n-1}_r,
			\]
			where $h_{ij}$ are the coefficients of $h$ in the given coordinate system, the comma denotes partial differentiation, $S^{n-1}_r$ is the coordinate sphere of radius $r$ in the asymptotic region and $\eta$ is its outward unit normal (with respect to the flat metric).  
			The problem remains of checking that the expression above does {\em not} depend on the particular coordinate system chosen near infinity. The first step in confirming this assertion involves the construction of {harmonic} coordinates; this is explained in \cite[Theorem 9.3]{lee1987yamabe}, which is a rather straightforward consequence of Theorem \ref{self:ac0the}.
		\end{remark}
		
		\subsection{The Laplacian in ${\rm ACyl}$ manifolds}\label{as:cyl:man}
		This class of manifolds appears in Example \ref{asym:cyl} above, so that $s=0$ in (\ref{met:edge}). Thus, Theorem \ref{self:adj} applies with $h=g_0$ and  $a=0$.  If $x=e^{-r}$ and $\delta=-\beta$ then the Sobolev-Mellin norm in (\ref{norm:near}) becomes
		\[
		\int|e^{\delta r} u(r,z)|^pd{\rm vol}_{h}(r,z),
		\]
		which gives rise to the weighted Sobolev spaces $W^p_{\sigma,\delta}(V)$ considered in \cite{lockhart1985elliptic}, but notice  that these authors use $\log r$ instead of $r$. The following result is an immediate consequence of Remark \ref{a:0:2}.
		
		\begin{theorem}\label{self:ac0the:cyl}
			If $n\ge 4$ then the Laplacian map
			\[
			\Delta_{h,\delta}:W^p_{\sigma,\delta}(V)\to W^p_{\sigma-2,\delta}(V)
			\]
			is Fredholm if $0<\delta<\sqrt{\mu_{g_F}}$, where $\mu_{g_F}$ is the first (positive) eigenvalue of $\Delta_{g_F}$.
		\end{theorem}
		
		The H\"older counterpart of this result is used in \cite{haskins2015asymptotically} to study  asymptotically cylindrical Calabi-Yau manifolds.
		
		\subsection{The Laplacian in ${\rm AC}_0/{\rm AC}_\infty$ manifolds}\label{man:0inf} 	This class of manifolds appears in Example \ref{a0:ainf} above, so that $s=\pm 1$ in (\ref{met:edge}) depending on the nature of the end. The corresponding mapping properties for the Laplacian are formulated in weighted Sobolev spaces incorporating the norms induced by  (\ref{norm:0}) and (\ref{norm:1}) above.  These properties, including the extra information coming from the jumps in the Fredholm index, are used in \cite{pacini2013special} to study the moduli space of special Lagrangian conifolds in $\mathbb C^m$; see also \cite{joyce2003special}.
		
		\subsection{The Dirac operator in ${\rm AC}_0$ spin manifolds}\label{dir:ac0:man} For this class of manifolds, Theorems \ref{albin:mellin} and \ref{albin:mellin:enh} apply with $s=1$ (and $h=g_1$) so if we further assume that $\kappa_{h}|_{U'}\geq 0$ then the Dirac operator $\dirac_{h}$ is Fredholm of index $0$ for $0<\beta<n-1$ with the core Dirac being essentially self-adjoint for $\beta=n/2$. Now recall that if $n$ is even then the spinor bundle decomposes as $S_V=S_V^+\oplus S_V^-$, with a corresponding decomposition for $\dirac_{h}$:
		\[
		\dirac_{h}=\left(
		\begin{array}{cc}
			0 & \dirac_{h}^-\\
			\dirac_{h}^+ & 0
		\end{array}
		\right)
		\]  
		where $\dirac_{h}^\pm:\Gamma(S_V^\pm)\to \Gamma(S_V^\mp)$, the chiral Dirac operators, are adjoint to each other. Thus, it makes sense to consider the {\em index} of $\dirac_{h}^+$:
		\[
		{\rm ind}\, \dirac_{h}^+=\dim\ker \dirac_{h}^+-\dim\ker \dirac_{h}^-.
		\] 
		This fundamental integer invariant can be explicitly computed in terms of topological/geometric data of the underlying ${\rm AC}_0$ manifold by means of heat asymptotics \cite{albin2016index,chou1985dirac,lesch1997differential}. The resulting formula has been used in \cite{de2022scalar} to exhibit obstructions for the existence of conical metrics with positive scalar curvature. 
		
		\subsection{The Dirac operator in asymptotically flat spin manifolds}\label{dir:asym:f}
		If an asymptotically flat manifold $V$ as in Remark \ref{as:anal} (with $h=g_{-1}$) is spin and satisfies $\kappa_{h}\geq 0$ in the asymptotic region then Theorem \ref{albin:mellin:enh} applies and $\dirac_{h}:L^p_{\sigma,\beta}(S_V)\to L^p_{\sigma-1,\beta-1}(S_V)$ is Fredholm of index $0$ for $1-n<\beta<0$. If we further assume that $\kappa_{h}\geq 0$ {\em everywhere} then integration by parts starting with the Weitzenb\"ock formula for the Dirac Laplacian $\dirac_{h}^2$ shows that $\dirac_{h}$ is injective and hence surjective by Fredholm alternative. We now take a {\em parallel} spinor $\phi_{\infty}$ in $\mathbb R^n$, $|\phi_\infty|=1$, and transplant it to the asymptotic region by means of the diffeomorphism $\psi$ in Example \ref{ex:af}. If we still denote by $\phi_{\infty}$ a smooth extension of this spinor to the whole of $V$, then 
		a computation shows that, as $r\to\infty$, 
		\[
		\dirac_{h}\phi_\infty=O(|\partial h|)=O(r^{-\nu_\infty-1})\in L^p_{\sigma-1,\beta-1}(S_V), \quad \beta\in\left[1-\frac{n}{2},0\right),
		\]
		so there exists $\phi_0\in L^p_{\sigma,\beta}(S_V)$ with $\dirac_{h}\phi_0=-\dirac_{h}\phi_\infty$. It follows that $\phi=\phi_0+\phi_\infty$ is harmonic ($\dirac_{h}\phi=0 $) and $|\phi-\phi_\infty|=O(r^{\beta})$. 
		With this spinor $\phi$ at hand,  another (more involved!) integration by parts     yields Witten's remarkable formula for the ADM mass of $(V,h)$:
		\[
		\mathfrak m_{(V,h)}=c_n\int_V\left(|\nabla\phi|^2+\frac{\kappa_h}{4}|\phi|^2\right)d{\rm vol}_h, \quad c_n>0.
		\]
		From this we easily deduce the following fundamental positive mass inequality.
		
		\begin{theorem}\cite{witten1981new} If $(V,h)$ is asymptotically flat and spin as above and $\kappa_h\geq 0$ everywhere then $\mathfrak m_{(V,h)}\geq 0$. Moreover, the equality holds only if $(V,h)=(\mathbb R^n,\delta)$ isometrically.
		\end{theorem}	
		
		The details of the argument above may be found in \cite[Appendix]{lee1987yamabe}.

		\section{Further applications}\label{nonempty:bd}
		
		The techniques described above may be adapted to handle more general situations. We only briefly discuss here three interesting cases.
		
		\subsection{Conformally conical manifolds with boundary} Here we consider conformally conical manifolds carrying a non-empty boundary $\partial X$ which is allowed to reach the tip of the cone. The formal definition is as in Example \ref{ex:conic}, except that the link $F$ itself carries a non-empty boundary $\partial F$. The key observation now is that both $\partial X$ and the double $2X$ along the boundary $\partial X$ are conformally conical manifolds as in Definition \ref{conic:metric} (the links of these ``boundaryless'' manifolds are $\partial F$ and $2F$, respectively). Thus, we are led to ask whether the Calder\'on-Seeley technique mentioned in the Introduction (for smooth manifolds) may be adapted to this context. This program has been carried out in \cite{coriasco2007realizations}, where it is shown, among other things, that the realizations of the Laplacian under standard boundary conditions (Dirichlet/Neumann) may be treated as well, at least in the ``straight'' case where the link metric is not allowed to vary with $x$ \cite[Section 5]{coriasco2007realizations}; see also \cite{fritzsch2020calder} for an approach in the setting of fibred cusp operators. If we invert the conical singularity as in Example $\ref{ex:af}$, we obtain an asymptotically flat manifold with a {\em non-compact} boundary and this theory provides a (rather sophisticated) approach to the results obtained ``by hand'' in \cite[Appendix A]{almaraz2014positive}. Finally, we mention that the setup in \cite{coriasco2007realizations} also applies to the realization of the Dirac operator acting on spinors under MIT bag boundary condition, so after inversion we recover the analytical machinery underpinning the positive mass theorems for asymptotically flat initial data sets in \cite{almaraz2014positive,almaraz2021spacetime}.

		\subsection{{Asymptotically hyperbolic} spaces}\label{sub:edge}
		We may consider an edge space $(X,g_s)$ with $g_s={x^{2s-2}}(\overline g+o(1))$ and $\overline g$ as in (\ref{met:edge:ex}). Here, 
		\[
		x^{2s}\Delta_{g_s}|_{U'}=\mathsf D^2+\tilde a\mathsf D+\Delta_{g_F}+x^2\Delta_{g_Y}+o(1), \quad \tilde a=a-d,\quad d=\dim Y.  
		\]
		If we specialize to the ``pure'' edge case in which the cone fiber $\mathcal C^F$ degenerates into a line ($F$ becomes a point) then $d=n-1$ and $\tilde a=a+1-n$, $n= \dim X\geq 3$,  and 
		\[
		g_s|_{U'}=x^{2s-2}\left(\overline g+o(1)\right),  \quad \overline g=dx^2+g_Y.
		\] 
		If we further take $s=0$ then  $(X,g_0)$ is {\em conformally compact} with $(Y,[g_Y])$ as its {\em conformal boundary} and $x$ is the corresponding {\em defining function}. {In particular, since $|dx|_{\overline g}=1$ along $Y$, a computation shows that $g_0$ is {\em asymptotically hyperbolic} in the sense that its sectional curvature approaches $-1$ as $x\to 0$.} Since $a=(n-2)s=0$, the conormal symbol gets replaced by
		\begin{equation}\label{con:symb:cc}
			\xi_{\Delta_{g_0}}(\zeta)=\zeta^2+(n-1)\zeta, 
		\end{equation}
		whose roots define a {\em unique} interval $(1-n,0)$ where the weight parameter $\beta$ is allowed to vary.
		Here, it is convenient to set 
		\[
		\beta=-\delta+\frac{1-n}{p}, \quad p>1,
		\]	
		so the Sobolev-Mellin norm in (\ref{norm:near}) becomes
		\[
		\int |x^{-\delta} u(x,y)|^px^{-n}dxd{\rm vol}_{g_Y}=\int |x^{-\delta} u(x,y)|^pd{\rm vol}_{g_0},
		\]
		which defines the weighted Sobolev spaces $H^{\sigma,p}_\delta(X)$ considered in \cite{andersson1993elliptic,lee2006fredholm}.  
		A  variation of the procedure above then yields the following result, which should be compared to  \cite[Proposition F]{lee2006fredholm} and \cite[Corollary 3.13]{andersson1993elliptic}; this latter reference only treats the case $p=2$. 
		
		\begin{theorem}\label{self:cc}
			The Laplacian map
			\begin{equation}\label{lap:cc}
				\Delta_{g_0,\delta}:H^{\sigma,p}_{\delta}(X)\to H^{\sigma-2,p}_{\delta}(X)
			\end{equation}
			is Fredholm of index $0$ if 
			\begin{equation}\label{int:hyp}
				\frac{1-n}{p}<\delta<\frac{(n-1)(p-1)}{p}.
			\end{equation}
		\end{theorem}
		
		\begin{remark}\label{outside}
			Since in this {asymptotically hyperbolic} case the Laplacian of the total space of the restricted fiber bundle $F\times Y=\{{\rm pt}\}\times Y{\to} Y$ endowed with the metric $g_F\oplus g_Y=g_Y$ does not show up in (\ref{con:symb:cc}), we are led to suspect that (\ref{lap:cc}) fails to be Fredholm if $\delta$ does not satisfy (\ref{int:hyp}). 
			This is the case indeed  and a proof of this claim may be found in \cite{lee2006fredholm}. Notice also that for $p=2$, (\ref{int:hyp}) becomes $|\delta|^2<(n-1)^2/4$, a bound that also appears in MacKean's estimate \cite{mckean1970upper}, which in particular provides a sharp lower bound for the bottom of the spectrum of the Laplacian in the model space (this is of course hyperbolic $n$-space $\mathbb H^n$, which is obtained in the formalism above by taking $(Y,g_Y)$ to be a round sphere). In fact, asymptotic versions of this estimate are used in \cite{andersson1993elliptic} as a key ingredient in directly establishing the mapping properties of geometric operators in {asymptotically hyperbolic spaces}. 
		\end{remark}
		
		{	
			\begin{remark}\label{qing}
				In an asymptotically hyperbolic manifold as above, the operator $\mathcal L_{g_0}^{(t)}:=\Delta_{g_0}+t(n-1-t)$, $t\in\mathbb C$, whose conormal symbol is
				\[
				\xi_{\mathcal L^{(t)}_{g_0}}(\zeta)=\zeta^2+(n-1)\zeta+t(n-1-t),
				\]
				also plays a distinguished role \cite{mazzeo1987meromorphic,graham2003scattering,case2016fractional}. 
				Let us assume that $2t\in\mathbb R\backslash \{n-1\}$ so that, by symmetry, we may take $2t>n-1$. 
				Proceeding as above, we see that 
				\[
				\mathcal L^{(t)}_{g_0,\delta}:H^{\sigma,p}_{\delta}(X)\to H^{\sigma-2,p}_{\delta}(X)
				\]
				is Fredholm of index $0$ if
				\[
				\frac{(n-1-t)p+1-n}{p}<\delta < 	\frac{tp+1-n}{p}. 
				\]
				If $g_0$ is Einstein (${\rm Ric}_{g_0}=-(n-1)g_0$), the special choice $t=n$ is particularly important: the so-called {\em static potentials} (that is, solutions of $\nabla^2_{g_0}V=Vg_0$) all lie in the kernel of $\mathcal L^{(n)}_{g_0}$.	
				In this case, the H\"older counterpart of this result (which is obtained by sending $p\to +\infty$ in the assertion above) has been used in \cite{qing2003rigidity} to establish the existence of ``approximate'' static potentials. As a consequence, an asymptotically hyperbolic Einstein manifold with the round sphere as its conformal infinity was shown to be isometric to hyperbolic $n$-space, at least if $4\leq n\leq 7$.
			\end{remark}
		}
		
		\subsection{{Asymptotically hyperbolic} spaces with boundary}\label{further:cc}
		A more general kind of {asymptotically hyperbolic} space is obtained by assuming that the underlying conformally compact space $X$ carries a boundary decomposing as $\partial X=Y \cup Y_{\rm f}$, with the intersection $\Sigma=Y\cap Y_{\rm f}$ being a (intrinsically smooth) co-dimension two corner. We assume further that there exists a tubular neighborhood $U$ of $Y$ on which a defining function $x$ for $Y$ exists so that 
		\[
		g|_U=x^{-2}(\overline g+o(1)),\quad \overline g=dx^2+g_Y(y), 
		\]
		where $g_Y$ is a metric in $Y$. Thus, the conformal boundary $(Y,[g_Y])$ itself carries a boundary, namely, $(\Sigma,[{g_Y}|_\Sigma])$, whereas the other piece of the boundary, $Y_{\rm f}$, remains at a finite distance. Finally, we impose that $\nabla_{\overline g}x$ is tangent to $Y_{\rm f}$ along $U\cap Y_{\rm f}$, so that $Y$ and $Y_{\rm f}$ meet orthogonally along $\Sigma$. It is immediate to check that: i) the ``finite'' boundary $\partial X_{\rm f}:=Y_{\rm f}$ is a pure edge space as above with conformal infinity $(\Sigma,[g_Y|_\Sigma])$; ii) the double $2X_{\rm f}:=X\sqcup_{Y_{\rm f}}-X$  is naturally a pure edge space as above with conformal boundary having the closed manifold $2Y$ as carrier and conformal structure induced by $g_Y$. Thus, at least in principle, the appropriate version of the Calder\'on-Seeley approach mentioned in the Introduction should apply here. Very likely, this follows from the corresponding adaptation of the general setup in \cite{coriasco2007realizations,fritzsch2020calder}, so that both the realizations of the Laplacian and the Killing Dirac operator (under Dirichlet/Neumann boundary and chiral boundary conditions, respectively, imposed along $\partial X_{\rm f}$) may be shown to be Fredholm in suitable Sobolev scales. This should provide an alternate approach to the analysis underlying the positive mass theorems for asymptotically hyperbolic initial data sets in \cite{almaraz2020mass,almaraz2021spacetime}.

		\bibliographystyle{alpha}
		\bibliography{tribuzy-mat-cont}	

\newcommand{\etalchar}[1]{$^{#1}$}
\begin{thebibliography}{BHM{\etalchar{+}}15}

\bibitem[AB64]{atiyah1964index}
Michael~F. Atiyah and Raoul Bott.
\newblock The index problem for manifolds with boundary.
\newblock In {\em Bombay Colloquium on Differential Analysis}, volume 175, page
  186. Oxford, 1964.

\bibitem[ABdL16]{almaraz2014positive}
S{\'e}rgio Almaraz, Ezequiel Barbosa, and Levi~Lopes de~Lima.
\newblock A positive mass theorem for asymptotically flat manifolds with a
  non-compact boundary.
\newblock {\em Communications in Analysis and Geometry}, 24(4):673--715, 2016.

\bibitem[AdL20]{almaraz2020mass}
S{\'e}rgio Almaraz and Levi~Lopes de~Lima.
\newblock The mass of an asymptotically hyperbolic manifold with a non-compact
  boundary.
\newblock {\em Annales Henri Poincar{\'e}}, 21(11):3727--3756, 2020.

\bibitem[AdLM21]{almaraz2021spacetime}
S{\'e}rgio Almaraz, Levi~Lopes de~Lima, and Luciano Mari.
\newblock Spacetime positive mass theorems for initial data sets with
  non-compact boundary.
\newblock {\em International Mathematics Research Notices}, 2021(4):2783--2841,
  2021.

\bibitem[AGR16]{albin2016index}
Pierre Albin and Jesse Gell-Redman.
\newblock The index of {D}irac operators on incomplete edge spaces.
\newblock {\em SIGMA. Symmetry, Integrability and Geometry: Methods and
  Applications}, 12:089, 2016.

\bibitem[ALMP12]{albin2012signature}
Pierre Albin, {\'E}ric Leichtnam, Rafe Mazzeo, and Paolo Piazza.
\newblock The signature package on {W}itt spaces.
\newblock In {\em Annales scientifiques de l'Ecole normale sup{\'e}rieure},
  volume~45, pages 241--310, 2012.

\bibitem[And93]{andersson1993elliptic}
Lars Andersson.
\newblock Elliptic systems on manifolds with asymptotically negative curvature.
\newblock {\em Indiana University Mathematics Journal}, pages 1359--1388, 1993.

\bibitem[APS75]{atiyah1975spectral}
Michael~F. Atiyah, Vijay~K. Patodi, and Isadore~M. Singer.
\newblock Spectral asymmetry and riemannian geometry. i.
\newblock {\em Mathematical Proceedings of the Cambridge Philosophical
  Society}, 77(1):43--69, 1975.

\bibitem[Bar86]{bartnik1986mass}
Robert Bartnik.
\newblock The mass of an asymptotically flat manifold.
\newblock {\em Communications on Pure and Applied Mathematics}, 39(5):661--693,
  1986.

\bibitem[BHM{\etalchar{+}}15]{bourguignon2015spinorial}
Jean-Pierre Bourguignon, Oussama Hijazi, Jean-Louis Milhorat, Andrei Moroianu,
  and Sergiu Moroianu.
\newblock {\em A spinorial approach to Riemannian and conformal geometry}.
\newblock European Mathematical Society, 2015.

\bibitem[CAC16]{case2016fractional}
Jeffrey~S. Case and Sun-Yung Alice~Chang.
\newblock On fractional {GJMS} operators.
\newblock {\em Communications on Pure and Applied Mathematics},
  69(6):1017--1061, 2016.

\bibitem[Che79]{cheeger1979spectral}
Jeff Cheeger.
\newblock On the spectral geometry of spaces with cone-like singularities.
\newblock {\em Proceedings of the National Academy of Sciences},
  76(5):2103--2106, 1979.

\bibitem[Cho85]{chou1985dirac}
Arthur~W. Chou.
\newblock The {D}irac operator on spaces with conical singularities and
  positive scalar curvatures.
\newblock {\em Transactions of the American Mathematical Society},
  289(1):1--40, 1985.

\bibitem[CP11]{chazarain2011introduction}
Jacques Chazarain and Alain Piriou.
\newblock {\em Introduction to the theory of linear partial differential
  equations}.
\newblock Elsevier, 2011.

\bibitem[CSCB79]{chaljub1979problemes}
Alice Chaljub-Simon and Yvonne Choquet-Bruhat.
\newblock Probl{\`e}mes elliptiques du second ordre sur une vari{\'e}t{\'e}
  euclidienne {\`a} l'infini.
\newblock {\em Annales de la Facult{\'e} des sciences de Toulouse:
  Math{\'e}matiques}, 1(1):9--25, 1979.

\bibitem[CSS07]{coriasco2007realizations}
Sandro Coriasco, Elmar Schrohe, and J\"org Seiler.
\newblock Realizations of differential operators on conic manifolds with
  boundary.
\newblock {\em Annals of Global Analysis and Geometry}, 31(3):223--285, 2007.

\bibitem[dL22]{de2022scalar}
Levi~Lopes de~Lima.
\newblock The scalar curvature in conical manifolds: some results on existence
  and obstructions.
\newblock {\em Annals of Global Analysis and Geometry}, 61(3):641--661, 2022.

\bibitem[ES12]{egorov2012pseudo}
Iouri Egorov and Bert-Wolfgang Schulze.
\newblock {\em Pseudo-differential operators, singularities, applications},
  volume~93.
\newblock Birkh{\"a}user, 2012.

\bibitem[FGS20]{fritzsch2020calder}
Karsten Fritzsch, Daniel Grieser, and Elmar Schrohe.
\newblock The {C}alder\'on projector for fibred cusp operators.
\newblock {\em arXiv:2006.04645}, 2020.

\bibitem[Fri00]{friedrich2000dirac}
Thomas Friedrich.
\newblock {\em Dirac operators in Riemannian geometry}, volume~25.
\newblock American Mathematical Soc., 2000.

\bibitem[GKM07]{gil2007geometry}
Juan~B. Gil, Thomas Krainer, and Gerardo~A. Mendoza.
\newblock Geometry and spectra of closed extensions of elliptic cone operators.
\newblock {\em Canadian Journal of Mathematics}, 59(4):742--794, 2007.

\bibitem[Gri01]{grieser2001basics}
Daniel Grieser.
\newblock Basics of the $b$-calculus.
\newblock In {\em Approaches to singular analysis}, pages 30--84. Springer,
  2001.

\bibitem[Gru12]{grubb2012functional}
Gerd Grubb.
\newblock {\em Functional calculus of pseudodifferential boundary problems},
  volume~65.
\newblock Springer Science \& Business Media, 2012.

\bibitem[GZ03]{graham2003scattering}
Charles~R. Graham and Maciej Zworski.
\newblock Scattering matrix in conformal geometry.
\newblock {\em Inventiones mathematicae}, 152(1):89--118, 2003.

\bibitem[HHN15]{haskins2015asymptotically}
Mark Haskins, Hans-Joachim Hein, and Johannes Nordstr{\"o}m.
\newblock Asymptotically cylindrical {C}alabi--{Y}au manifolds.
\newblock {\em Journal of Differential Geometry}, 101(2):213--265, 2015.

\bibitem[HLV18]{hartmann2018domain}
Luiz Hartmann, Matthias Lesch, and Boris Vertman.
\newblock On the domain of {D}irac and {L}aplace type operators on stratified
  spaces.
\newblock {\em Journal of Spectral Theory}, 8(4):1295--1348, 2018.

\bibitem[H{\"o}r07]{hormander2007analysis}
Lars H{\"o}rmander.
\newblock {\em The analysis of linear partial differential operators III:
  Pseudo-differential operators}.
\newblock Springer Science \& Business Media, 2007.

\bibitem[Joy03]{joyce2003special}
Dominic Joyce.
\newblock Special lagrangian submanifolds with isolated conical singularities.
  {V}. {S}urvey and applications.
\newblock {\em Journal of Differential Geometry}, 63(2):279--347, 2003.

\bibitem[KM18]{krainer2018friedrichs}
Thomas Krainer and Gerardo~A. Mendoza.
\newblock The {F}riedrichs extension for elliptic wedge operators of second
  order.
\newblock {\em Advances in Differential Equations}, 23(3/4):295--328, 2018.

\bibitem[Lau03]{lauter2003pseudodifferential}
Robert Lauter.
\newblock {\em Pseudodifferential analysis on conformally compact spaces}.
\newblock American Mathematical Soc., 2003.

\bibitem[Lee06]{lee2006fredholm}
John~M. Lee.
\newblock {\em Fredholm operators and Einstein metrics on conformally compact
  manifolds}, volume~13.
\newblock American Mathematical Soc., 2006.

\bibitem[Les97]{lesch1997differential}
Matthias Lesch.
\newblock {\em Differential operators of Fuchs type, conical singularities, and
  asymptotic methods}, volume 136 of {\em Teubner Texte zur Mathematik}.
\newblock Teubner--Verlag, 1997.

\bibitem[LMO85]{lockhart1985elliptic}
Robert~B. Lockhart and Robert~C. Mc~Owen.
\newblock Elliptic differential operators on noncompact manifolds.
\newblock {\em Annali della Scuola Normale Superiore di Pisa-Classe di
  Scienze}, 12(3):409--447, 1985.

\bibitem[LP87]{lee1987yamabe}
John~M. Lee and Thomas~H. Parker.
\newblock The {Y}amabe problem.
\newblock {\em Bulletin (New Series) of the American Mathematical Society},
  17(1):37--91, 1987.

\bibitem[LS01]{lauter2001pseudodifferential}
Robert Lauter and J{\"o}rg Seiler.
\newblock Pseudodifferential analysis on manifolds with boundary - a comparison
  of $b$-calculus and cone algebra.
\newblock In {\em Approaches to Singular Analysis}, pages 131--166. Springer,
  2001.

\bibitem[Maz91]{mazzeo1991elliptic}
Rafe Mazzeo.
\newblock Elliptic theory of differential edge operators {I}.
\newblock {\em Communications in Partial Differential Equations},
  16(10):1615--1664, 1991.

\bibitem[McK70]{mckean1970upper}
Henry~P. McKean.
\newblock An upper bound to the spectrum of ${\Delta}$ on a manifold of
  negative curvature.
\newblock {\em Journal of Differential Geometry}, 4(3):359--366, 1970.

\bibitem[Mel90]{melrose1990pseudodifferential}
Richard~B. Melrose.
\newblock Pseudodifferential operators, corners and singular limits.
\newblock In {\em Proc. Int. Congress of Mathematicians, Kyoto, Japan}.
  Springer, 1990.

\bibitem[Mel93]{melrose1993atiyah}
Richard~B. Melrose.
\newblock {\em The Atiyah-Patodi-Singer index theorem}.
\newblock CRC Press, 1993.

\bibitem[Mel96]{melrose1996differential}
Richard~B. Melrose.
\newblock {\em Differential analysis on manifolds with corners}.
\newblock in preparation, 1996.

\bibitem[MM87]{mazzeo1987meromorphic}
Rafe~R Mazzeo and Richard~B Melrose.
\newblock Meromorphic extension of the resolvent on complete spaces with
  asymptotically constant negative curvature.
\newblock {\em Journal of Functional Analysis}, 75(2):260--310, 1987.

\bibitem[Pac13a]{pacini2010desingularizing}
Tommaso Pacini.
\newblock Desingularizing isolated conical singularities: uniform estimates via
  weighted sobolev spaces.
\newblock {\em Communications in Analysis and Geometry}, 21(1):105--170, 2013.

\bibitem[Pac13b]{pacini2013special}
Tommaso Pacini.
\newblock Special {L}agrangian conifolds, {I}: moduli spaces.
\newblock {\em Proceedings of the London Mathematical Society},
  107(1):198--224, 2013.

\bibitem[Qin03]{qing2003rigidity}
Jie Qing.
\newblock On the rigidity for conformally compact {E}instein manifolds.
\newblock {\em International Mathematics Research Notices},
  2003(21):1141--1153, 2003.

\bibitem[RS13]{roidos2013cahn}
Nikolaos Roidos and Elmar Schrohe.
\newblock The {C}ahn-{H}illiard equation and the {A}llen-{C}ahn equation on
  manifolds with conical singularities.
\newblock {\em Communications in Partial Differential Equations},
  38(5):925--943, 2013.

\bibitem[Sch98]{schulze1998boundary}
Bert-Wolfgang Schulze.
\newblock {\em Boundary value problems and singular pseudo-differential
  operators}.
\newblock Pure and Applied Mathematics Interscience Series of Texts,
  Monographs, and Tracks. John Wiley, 1998.

\bibitem[SS01]{schrohe1999ellipticity}
Elmar Schrohe and J\"org Seiler.
\newblock Ellipticity and invertibility in the cone algebra on
  ${L}_p$-{S}obolev spaces.
\newblock {\em Integral Equations and Operator Theory}, (41):93--114, 2001.

\bibitem[Wit81]{witten1981new}
Edward Witten.
\newblock A new proof of the positive energy theorem.
\newblock {\em Communications in Mathematical Physics}, 80(3):381--402, 1981.

\bibitem[WRL95]{wloka1995boundary}
Joseph~T. Wloka, Brian Rowley, and Bohdan Lawruk.
\newblock {\em Boundary value problems for elliptic systems}.
\newblock Cambridge University Press, 1995.

\end{thebibliography}
		
			%
			%

	\end{document}